\documentclass[11pt,reqno]{amsart}
\usepackage{amssymb,amsmath,amscd,mathrsfs}
\usepackage{color}

\usepackage{graphicx}
\usepackage{epstopdf}

\theoremstyle{plain}
\newtheorem{theorem}{Theorem}[section]

\newtheorem{lemma}[theorem]{Lemma}
\newtheorem{proposition}[theorem]{Proposition}
\newtheorem{corollary}[theorem]{Corollary}

\theoremstyle{definition}
\newtheorem{definition}[theorem]{Definition}
\newtheorem{remark}[theorem]{Remark}

\newtheorem{question}[theorem]{Question}

\newtheorem{example}[theorem]{Example}

\numberwithin{equation}{section}

\newcommand\des{\mathop{\rm des}}

\newcommand{\Des}{\operatorname{Des}}

\newcommand{\la}{{\lambda}}
\newcommand{\g}{{\gamma}}

\newcommand{\be}{{\beta}}

\newcommand\nneg{\mathop{ \rm neg}}

\def\ZZ{{\mathbb Z}}
\def\NN{{\mathbb N}}

\def\k+1th{(k+1)^{th}}

\DeclareMathOperator{\maj}{maj}

\DeclareMathOperator{\fmaj}{fmaj}

\DeclareMathOperator{\sgn}{sgn}

\def\maj{\mathop{\rm maj}\nolimits}

\def\inv{\mathop{\rm inv}\nolimits}
\def\des{\mathop{\rm des}\nolimits}

\def\il{\bigl]\kern-.55em\bigl]}
\def\ir{\bigr]\kern-.55em\bigr]}

\title{On some analogues of Carlitz's identity for the hyperoctahedral group}
\author{Riccardo Biagioli and Jiang Zeng}
\address{Universit\'e de Lyon, Universit\'e Lyon 1, Institut Camille
Jordan, UMR 5208 du CNRS, F-69622, Villeurbanne Cedex, France}
\email{biagioli@math.univ-lyon1.fr, zeng@math.univ-lyon1.fr}
\thanks{The two authors are supported by
the grant  ANR-08-BLAN-0243-03 and by the Program P2R Franco-Isra\'elien en Math\'ematiques}
\begin{document}
\begin{abstract}
We give a new description of the flag major index, introduced by Adin and Roichman, 
by using a major index defined by Reiner. This allows us to 
establish a connection between an identity  of Reiner and some more recent results due to Chow and Gessel.
Furthermore we generalize the main identity of Chow and Gessel  by computing the four-variate 
generating series of descents, major index, length, and number of negative entries over Coxeter groups of type $B$ and $D$.  
\end{abstract}
\maketitle

\section{Introduction}

The following identity due to Carlitz~\cite{ca75} provides a remarkable $q$-analogue of a well-known identity for the Eulerian polynomials
\begin{align}\label{eq:carlitz} 
\frac{\sum_{\sigma\in S_n} t^{{\rm des}(\sigma)}q^{{\rm maj}(\sigma)}}{\prod_{i=0}^{n}(1-tq^i)}=\sum_{k\geq0}(1+q+\cdots +q^k)^nt^k.
\end{align}
Here $S_n$ denotes the symmetric group of order $n$,  and  ``maj'' and ``des''  the number of descents and the major index, respectively. It is easy to see that Equation~\eqref{eq:carlitz} is equivalent to the following one
\begin{align}\label{eq:carlitzbis}
\sum_{n\geq 0}\frac{\sum_{\sigma\in S_n} t^{{\rm des}(\sigma)}q^{{\rm maj}(\sigma)}}{\prod_{i=0}^{n}(1-tq^i)}\frac{u^n}{n!}=\sum_{k\geq0}t^k\exp{(1+q+\cdots +q^k)u}. 
\end{align}
At the end of 1970's,
Gessel~\cite{ge-thesis},  and Garsia and Gessel~\cite{gg79} gave substantial extensions of \eqref{eq:carlitzbis}. 
Furthermore  Reiner \cite{re93, re93bis} generalized Garsia and Gessel's work to  the hyperoctahedral group, denoted  by  $B_n$. 
In particular, he obtained the following identity 
\begin{align}\label{eq:reiner1}
 \sum_{n\geq 0}\frac{\sum_{\be\in B_n}
t^{d_{R}(\be)}q^{\maj_R(\be)}p^{\inv_{R}(\be)}a^{\nneg(\be)}}{\prod_{i=0}^n(1-tq^i)}  \frac{u^n}{[\hat n]_{a,p}!}=
\sum_{k\geq 0}t^k{\hat e}[u]_{a,p}e[qu]_{p}\cdots e[q^ku]_{p},
\end{align}
from which  a $B_n$-analogue of \eqref{eq:carlitz} can be derived.
The definitions of the above  statistics as well as  all undefined notation will be given in the following sections. 

Motivated by their work on invariant algebras, Adin and Roichman introduced in~\cite{ar01} the flag major index ``$\fmaj$''. 
This statistic turned out to be the key ingredient of new $B_n$-analogues of Carlitz identity (see \cite[Problem  1.1]{abr03} on Foata's problem).
Two of such analogues were given by  Adin, Brenti, and Roichman~\cite[Theorem 4.2, Corollary 4.5]{abr03}. More recently, a third one was proposed by Chow and Gessel~\cite[Theorem 3.7]{cg07}, which in equivalent form reads as follows~\cite[Theorem 3.8]{cg07}
\begin{align}\label{eq:cg-carlitzbis}
\sum_{n\geq 0}\frac{\sum_{\beta \in B_n}t^{\des_B(\beta)}q^{\fmaj(\beta)}}{\prod_{i=0}^n(1-tq^{2i})}\frac{u^n}{n!}=\sum_{k\geq 0}t^k \exp(1+q+\cdots +q^{2k})u.
\end{align}

Our starting point is the observation that the
substitution $q\leftarrow q^2$, $a\leftarrow q^{-1}$, $p\leftarrow 1$, and $u\leftarrow (1+q^{-1})u$ in Reiner's identity \eqref{eq:reiner1} yields
 \begin{align}\label{eq:reiner2}
 \sum_{n\geq 0}\frac{ \sum_{\be\in B_n}t^{d_{R}(\be)}q^{2 \maj_R(\be)-\nneg(\beta)}}{\prod_{i=0}^n(1-tq^{2i}) }\frac{u^n}{n!}
=\sum_{k\geq 0}t^k\exp(1+q+\cdots + q^{2k})u.
\end{align}
This implies immediately that the two pairs of statistics $({\rm des_B}, {\rm fmaj})$ and $(d_{R}, 2 \maj_R -\nneg)$ are equidistributed over $B_n$. 
Actually our first result (cf. Proposition~\ref{p:connection}) shows that these two pairs of statistics are equal over $B_n$. Therefore Equations \eqref{eq:cg-carlitzbis}  and \eqref{eq:reiner2}  are the same. This prompted us to look for refinements of Chow-Gessel identity with more parameters.

In Theorem~\ref{t:main1}, we compute the generating function of the four statistics $(\des_B, \maj, \ell_{B}, \nneg)$ over $B_{n}$.
It turns out that this distribution is different from that of Reiner's four statistics in \eqref{eq:reiner1} over $B_{n}$.  
In contrast to the method of Chow-Gessel~\cite{cg07}, that uses $q$-difference calculus and recurrence relations,
we adopt a more combinatorial approach inspired by that of Garsia-Gessel~\cite{gg79}, and Reiner~\cite{re93}. 
The basic idea is to encode signed sequences by pairs made of a signed permutation and a partition.

Finally we consider the Coxeter group of type $D$, and by using the same encoding technique 
 for the hyperoctahedral group, we enumerate its elements  by descents, major index, length, and number of negative entries (cf. Theorem~\ref{identityTD}).

\section{Preliminaries and notation}

In this section we give some definitions, notation and results that
will be used in the rest of this work. 
Let $P$ be a statement: the characteristic function $\chi$ of $P$ is defined as $\chi(P)=1$ if $P$ is true, and $0$ otherwise.
For $n \in \NN$ we let $[n]:= \{ 1,2, \ldots , n \} $ (where $[0]:= \emptyset $). Given $n, m \in \ZZ, \; n \leq m$, we let $[n,m]:=\{n,n+1, \ldots, m \}.$  
The cardinality of a set $A$ will be denoted either by $|A|$ or by $\#A$.
For $n\in \NN$, we let
\begin{eqnarray*}
(a;p)_n:=\left\{\begin{array}{ll} \;\; 1, & {\rm  if} \ n=0;\\
(1-a)(1-ap)\cdots (1-ap^{n-1}),& {\rm if} \ n\geq 1.
 \end{array}\right.
\end{eqnarray*}
For our study we need notation for $p$-analogs of integers and factorials. These are defined by the following expressions
\begin{align*}
[n]_p:&=1+p+p^2+\ldots + p^{n-1}, \\
[n]_p!:&=[n]_p[n-1]_p\cdots[2]_p[1]_p, \\
\ [\hat{n}]_{a,p}!:&=(-ap;p)_n [n]_p! ,\\
\end{align*}
where $[0]_{p}!=1$.  
For $n=n_0+n_1+\cdots +n_k$  with $n_{0}, \ldots, n_{k}\geq 0$ we define the $p$-multinomial coefficient by 
$$
{n\brack n_0,n_1,\ldots, n_k}_{p}:=\frac{[ n]_{p}!}{[n_0]_{p}![n_1]_p!\cdots [n_k]_p!}.
$$
Finally, 
$$e[u]_p:=\sum_{n\geq 0}\frac{u^n}{[n]_p !}, \quad {\rm and}  \quad
\hat e[u]_{a,p}:=\sum_{n\geq 0}\frac{u^n}{[\hat n]_{a,p}!}$$
are the two classical $p$-analogues of the exponential function.


We denote by $B_{n}$ the group of all bijections $\be$ of the set
$[-n,n]\setminus \{0\}$ onto itself such that
\[\be(-i)=-\be(i)\]
for all $i \in [-n,n]\setminus \{0\}$, with composition as the
group operation. This group is usually known as the group of {\em
signed permutations} on $[n]$, or as the {\em hyperoctahedral
group} of rank $n$.  If
$\be \in B_{n}$ then we write $\be=[\be(1),\dots,\be(n)]$ and we call this the
{\em window} notation of $\be$. 
We denote by $|\beta(i)|$ the absolute value of $\beta(i)$, and by $\sgn(\be(i))$ its sign.
For $\be \in B_n$ we let
\begin{eqnarray*}
\inv(\be) & := & |\{ (i,j) \in [n]\times [n] \mid \be(i)>\be(j) \}|\quad {\rm and} \\ 
\nneg(\be) & := &   |\{i \in [n] \mid \be(i)<0\}|.
\end{eqnarray*}

As set of generators for $B_n$ we
take $S_B:=\{s_1^B,\ldots,s_{n-1}^B,s_0^B\}$ where for $i
\in[n-1]$
\[s_i^B:=[1,\ldots,i-1,i+1,i,i+2,\ldots,n] \;\; {\rm and} \;\; s_0^B:=[-1,2,\ldots,n].\]
It is well-known that $(B_n,S_B)$ is a Coxeter system of type $B$ (see e.g., \cite[\S 8.1]{bb05}). The following characterization of the {\em length} function $\ell_B$ of $B_n$ with respect to $S_B$ is well-known~\cite[Proposition 8.1.1]{bb05}
$$
\ell_B(\be)=\inv(\be)-\sum_{\be(i)<0}\be(i).
$$
\begin{remark}
Note that the length statistic denoted $\inv_{R}$ in \eqref{eq:reiner1}
is  different from $\ell_{B}$, since it is computed with respect to a different set of generators, see \cite[\S 2]{re93} for its definition. However $\inv_{R}$ and $\ell_{B}$ are equidistributed over $B_{n}$. A refinement of this fact will be given after Equation~\eqref{e:tridistribution}.
\end{remark}

The {\em B-descent set} of $\be \in B_n$ is defined by 
\begin{eqnarray*}
\Des_B(\be)&:=&\{i \in [0,n-1] \mid \be(i) > \be(i+1)\},
\end{eqnarray*}
where  $\be(0):=0$, and its cardinality is denoted by $\des_B(\be)$. The inversion number $\inv$ and $\Des_B$ are computed by using the natural order 
$$ -n<\ldots <-1<0<1< \ldots < n $$ on the set $[-n,n]$. 
As usual the {\em major index}  is defined  to be the sum of descents
$$ \maj(\be)=\sum_{i\in \Des_B(\be)}i,$$
and the \emph{flag-major index}  \cite[(10)]{abr03} by $$\fmaj(\be):=2 \maj(\be)+ \nneg(\be).$$
We should point out that in the first paper \cite{ar01} Adin and Roichman defined the flag-major index by using a different order. Here we use their second definition which appears in the subsequent paper with Brenti \cite{abr03}, and that uses the natural order.
 
\smallskip
Reiner introduced the definition of descents in the usual geometric way  \cite[\S 2]{re93}. To give a combinatorial description of his definition 
 it is convenient to use the following order $<_{R}$
$$ 1<\ldots <n<_{ R}-n<\ldots <-1 $$ 
on the set $[-n,n]\setminus \{0\}$. By abuse of notation, we write $<$ instead of $<_{R}$ when they coincide.
Then Reiner's descent set reads  as follows
$$\Des_R(\be):=\{i \in [1,n] \mid \be(i) >_R \be(i+1)\},$$
where $\be(n+1):=n$. The cardinality of $\Des_R(\be)$ will be denoted by $d_{R}(\be)$.
To avoid confusion we denote  the major index associated to Reiner's descent set by $\maj_R$, i.e.,
$$\maj_R(\be)=\sum_{i\in \Des_R(\be)}i. $$

\begin{remark}\label{r:order}
The main difference between the two descent sets is that $0$ can be in $\Des_B$ but not in $\Des_R$,  while $n$ can be in 
$\Des_R$ but not in  $\Des_B$.
More precisely, $0\in \Des_B(\be)$ if and only if $\be(1)<0$, while $n \in \Des_R(\be)$ if and only if $\be(n)\in \{-1, -2,\ldots, -n\}$. Clearly, this reflects on the corresponding major indices. 
\end{remark}

\noindent  For example, if $\be=[-3,-4,1,6,-5,-2] \in B_6$ then 
$$\Des_B(\be)=\{0,1,4\}, \quad {\rm and} \quad \Des_R(\be)=\{1,2,6\}.$$
Hence $\maj(\be)=5$, while $\maj_R(\be)=9$.
\bigskip

We denote by $D_{n}$ the subgroup of $B_{n}$ consisting of all the
signed permutations having an even number of negative entries in
their window notation, more precisely
\[ D_{n} := \{\g \in B_{n} \mid \nneg(\g)\equiv 0 \; ({\rm mod} \; 2 ) \}. \]
It is usually called the {\em even-signed permutation group}. As a set of generators for $D_n$ we take
$S_D:=\{s_{0}^D,s_{1}^D,\dots,s_{n-1}^D\}$ where for $i \in [n-1]$
\[s_i^D:=s_i^B \;\; {\rm and} \;\; s_{0}^D:=[-2,-1,3,\ldots,n].\]

The following is a well-known combinatorial way to compute the length, and the descent set of $\g \in D_{n}$, (see, e.g., \cite[\S 8.2]{bb05})
\begin{eqnarray*}\label{lD} 
\ell_D(\g)&=&\ell_B(\g) - \nneg(\g),\; \; {\rm and} \\
\Des_D(\g)&=&\{i \in [0,n-1] \mid \g(i)>\g(i+1)\},
\end{eqnarray*}
where $\g(0):=-\g(2)$. The cardinality  of $\Des_D(\g)$ will be denoted by $\des_D(\g)$. 
\section{Connection between Reiner and Chow-Gessel's identities}

In the introduction we deduced that the two pairs of statistics $(\des_B,\fmaj)$ and $(d_{R}, 2 \maj- \nneg)$ are
equidistributed over $B_{n}$. In this section we show  that they are in fact identical.

\begin{proposition}\label{p:connection}
For any $\be \in B_n$ we have 
\begin{align}
d_{R}(\be)&=\des_B(\be),\label{eq:des}\\
\maj_{R}(\be)&= \maj(\be)+\nneg(\be).\label{eq:maj}
\end{align}
\end{proposition}
\begin{proof}
We first prove that $d_{R}(\be)=\des_B(\be)$ for every $\be \in B_n$. 
Clearly it suffices to  restrict our attention to the case of one descent. 
\begin{itemize}
\item[a)] If $i \in \Des_R(\be)$ is such that $\be(i)>\be(i+1)>0$ or $0>\be(i)>\be(i+1)$, then clearly $i\neq n$ and $i \in \Des_B(\be)$.
\item[b)] If $i \in \Des_R(\be)$ is such that $\be(i)<0$ and $\be(i+1)>0$, then there exists a $k$ such that 
$$
0\leq k \leq i-1, \quad 
0<\be(1)<\ldots<\be(k)\quad\textrm{and}\quad \be(k+1)<\ldots<\be(i)<0.
$$
If $k\geq 1$ then $k \in \Des_B(\be)$; if $k=0$ then  $\be(1)<0$ and so  
$0 \in \Des_B(\be)$. 
\end{itemize}

Now we consider  the major indices. It suffices to check the situation of 
one descent, since the general case follows easily by induction. In the first following three cases we let $i \in \Des_R(\be)$
be the only descent of $\be$ in $[1,n-1]$.
\begin{itemize}
\item[1)] Suppose that $\be(i)>\be(i+1)>0$. 
Then $0<\be(1)<\ldots<\be(i)$, and there are $h$ negative entries $(0 \leq h\leq n-i-1)$ such that
$$0<\be(i+1)<\ldots <\be(n-h) \quad {\rm and} \quad \be(n-h+1)<\ldots<\be(n)<0.$$
If $h=0$, then $\maj_R (\be) - \nneg(\be)=i=\maj(\be)$. If $h > 0$, then $\Des_R(\be)=\{i,n\}$ while $\Des_B(\be)=\{i,n-h\}$. Thus $ \maj_R (\be) - \nneg(\be)=(i+n)-h=\maj(\be)$.
\smallskip

\item[2)] Suppose that $0>\be(i)>\be(i+1)$. Then $\be(i+1)<\ldots<\be(n)<0$, and there are $k$ positive entries ($0\leq k\leq i-1$) such that
$$ 0<\be(1)<\ldots <\be(k)  \quad {\rm and} \quad \be(k+1)<\ldots<\be(i)<0.$$
In this case $\maj_R (\be)-\nneg(\be)=(i+n)-(n-k)=k+i=\maj(\be)$.
\smallskip

\item[3)] Suppose that $\be(i)<0$ and $\be(i+1)>0$. Then there exist $0\leq k\leq i-1$ and $0 \leq h\leq n-i-2$ such that
$$ 0<\be(1)<\ldots <\be(k)  \quad {\rm and} \quad\be(k+1)<\ldots<\be(i)<0, \quad {\rm and}$$
$$ 0<\be(i+1)<\ldots <\be(n-h)  \quad {\rm and} \quad \be(n-h+1)<\ldots<\be(n)<0.$$
If $h, k>0$, we obtain $\maj_R(\be)-\nneg(\be)=(i+n)-(h+i-k)=k+(n-h)=\maj(\be)$. The cases $h=0$ and $k=0$ are similar.
\smallskip

\item[4)] If $n \in \Des_R(\be)$ is the only descent then there exists a  $k$ $(0\leq k\leq n-1)$ such that 
$$
 0<\be(1)<\ldots <\be(k) \quad {\rm and} \quad \be(k+1)<\ldots<\be(n)<0.
$$
Hence $\maj_R(\be)-\nneg(\be)=n-(n-k)=k=\maj(\be)$.
\end{itemize}
\end{proof}
\begin{remark}
The first equality $d_{R}(\be)=\des_B(\be)$ can also be easily derived by using its  geometric interpretation.
\end{remark}

The next new description of the flag major index easily follows from Proposition~\ref{p:connection}.
\begin{corollary} \label{c:fmaj} 
For every $\be \in B_n$  we have 
$$\fmaj(\be)=2\maj_R(\be)-\nneg(\be).$$
\end{corollary}

\section{Encoding signed sequences }\label{bijection1}
In this section we introduce a  procedure encoding signed sequences by signed permutations and partitions, which will be 
used in the following  sections.
The basic idea can be found in Garsia and Gessel~\cite{gg79} and Reiner~\cite{re93}.

\medskip
Let ${\mathcal P}_n$ be the set of 
non decreasing sequences  of  nonnegative integers $(\lambda_1, \lambda_2, \ldots, \lambda_n)$,
 i.e., partitions  of length less than or equal to $n$.
 
\begin{definition}\label{d:bijection}
Given a sequence $f=(f_1,\ldots,f_n)\in \ZZ^n$ we define
a pair $(\beta,\lambda)\in B_n\times {\mathcal P}_n$, 
where:
\begin{itemize}
\item[a)] $\be$ is the unique signed permutation  satisfying 
the following three conditions:
\begin{enumerate}
\item[1)] $|f_{|\be(1)|}|\leq |f_{|\be(2)|}|\leq \ldots \leq |f_{|\be(n)|}|$,
\item[2)] $\sgn(\be(i)):=\sgn(f_{|\be(i)|})$,
\item[3)] If  $|f_{|\be(i)|}|=|f_{|\be(i+1)|}|$, then $\be(i)<\be(i+1)$;
\end{enumerate}
\item[b)]
$\lambda=(\la_1,\ldots,\la_n)$ is the partition with
$$\lambda_i:=|f_{|\be(i)|}|-|\{j\in {\Des}_B(\be) \mid j\leq i-1\}|\quad\text{for} \quad
1\leq i\leq n. $$
\end{itemize}
We denote $\pi(f):=\be$ and $\lambda(f):=\lambda$.
\end{definition}

\begin{remark} 
The above sequence $\la$ is clearly a partition since for all $i\in [n]$, $\lambda_i\geq 0$ and 
$$
\lambda_{i+1}-\lambda_i=
|f_{|\beta(i+1)|}|-|f_{|\be(i)|}|-\chi(i\in \Des_B(\pi))\geq 0 \quad \textrm{for $i\in [n-1]$}.
$$
Moreover, note that  if $i \in {\Des}_B(\be)$ then $|f_{|\be(i)|}|<|f_{|\be(i+1)|}|$, and 
$0\in {\Des}_B(\be)$ if and only if $f_{|\be(1)|}<0$.
\end{remark}

We introduce the following statistics on the set of signed sequences $\ZZ^n$.
\begin{definition}\label{d:statistics}
For $f=(f_1,\ldots,f_n) \in \ZZ^n$ we let
\begin{align*}
\max(f)&:=\max\{|f_i|\}, &  |f|&:=\sum_{i=1}^n |f_i|,\\
\ell_B(f)&:=\ell_B(\pi(f)),& \nneg(f)&:=|\{i \in [n] \mid f_i<0\}|.
\end{align*}
\end{definition}

\begin{proposition} \label{p:psi}
The map $\psi : \ZZ^n \rightarrow B_n\times {\mathcal P}_n$ defined by (see Definition~\ref{d:bijection})
$$f  \mapsto (\pi(f), \lambda(f))$$ 
is a bijection. Moreover, if we let $\be:=\pi(f)$ and $\la:=\la(f)$ then
\begin{align}
\max(f)&=\max(\lambda)+\des_B(\be),\label{eq:1}\\
|f|&=|\lambda|+n\des_B(\be) -\maj(\be).\label{eq:2}
\end{align}
\end{proposition}
\begin{proof}
To see that $\psi$ is a bijection we construct its inverse as follows. To each $(\be, \lambda)\in B_n\times {\mathcal P}_n$ associate the partition $\mu=(\mu_1,\ldots,\mu_n)$ where
$$\mu_i:= \lambda_i + |\{j \in {\Des}_B(\be) \mid j\leq i-1\}|,$$
and define the sequence $f=(f_1,\ldots, f_n) \in \ZZ^n$  by letting
$$
f_i:=\sgn(\be^{-1}(i))\, \mu_{|\be^{-1}(i)|}  \quad \textrm{for} \quad 1\leq i \leq n.
$$
It is easy to see that $\psi(f)=(\be, \lambda)$. Equations \eqref{eq:1} and \eqref{eq:2}
follow since $\la$ is a partition, and from
$$ 
\sum_{i=1}^n|\{j\in {\Des}_B(\be)|j\leq i-1\}|=\sum_{j\in {\Des}_B(\pi)}(n-j)=n\des_B(\be)-\maj(\be).
$$
\end{proof} 

\begin{example}
If $f=(-4,4,1,-3,6,3,-4)$ then 
$(|f_{|\be(1)|}|, \ldots, |f_{|\be(7)|}|)=(1,3,3,4,4,4,6),$ so 
$$\pi(f):=\be=(3,-4,6, -7, -1,2,5) \in B_7.$$
Moreover $\Des_B(\be)=\{1,3\}$, so $\lambda=(1,2,2,2,2,2,4)$.
We obtain $\max(f)=6$ and $\des_B(\be)=2$. Hence the identity \eqref{eq:2} reads $25=15+14-4$. 
Conversely, given the pair 
$$([5,-3,1,2,-4],(0,2,2,3,3))\in B_5\times {\mathcal P}_5,
$$
 we find
 $\mu=(0,3,3,4,5)$ and $f=(3,4,-3,-5,0)$.
\end{example}
\section{Main identity}
In this section we compute the generating function of the vector statistic 
$(\des_B, \maj, \ell_{B}, \nneg)$ over $B_{n}$.
 As special instances, we recover 
several  known identities listed after the proof. 
\begin{theorem}\label{t:main1}
 We have
\begin{align}\label{eq:main1}
 \sum_{n\geq 0}\frac{u^n}{(t;q)_{n+1}[\hat n]_{a,p}!}\sum_{\be\in B_n}
t^{\des_B(\be)}q^{\maj(\be)}p^{\ell_B(\be)}a^{\nneg(\be)}=
\sum_{k\geq 0}t^k\prod_{j=0}^{k-1}e[q^ju]_{p}\cdot {\hat e}[q^ku]_{a,p}.
\end{align}
\end{theorem}

\begin{proof}
The proof consists in computing in two different ways the series 
$$\sum_{f\in \ZZ^n} t^{\max(f)}q^{\max(f)\cdot n-|f|}
p^{\ell_B(f)}a^{\nneg(f)}.$$

We need the following preliminary result that can be proved as in  \cite[Lemma 3.1]{re93}.  
For $\underline{n}=(n_0,n_1,\ldots,n_k)\in \NN^{k+1}$ a composition of $n$, i.e., 
$n=n_0+\ldots+n_k$, we let
$$
\ZZ^n(\underline{n}):=\{
{f \in \ZZ^n \mid \#\{i:|f_i|=j\}=n_j}\}.
$$ 
Then
\begin{align}\label{lemmaReiner}
\sum_{f \in \ZZ^n(\underline{n})} 
p^{\ell_B(f)}a^{\nneg(f)}={n \brack n_0, n_1, \ldots, n_k}_{p}\frac{(-ap;p)_{n}}{(-ap;p)_{n_{0}}}.
\end{align}
A version of \eqref{lemmaReiner}
for the more general case of wreath products is given in \cite[Lemma 4.2]{bz}.

We can now start our computation.
\begin{align*}
\sum_{f\in \ZZ^n \mid \max(f)\leq k} q^{k\cdot n-|f|}
p^{\ell_B(f)}a^{\nneg(f)}
&=\sum_{n_0+\cdots+n_k=n}q^{k \sum_i n_i -\sum_i i n_i}
\sum_{f\in \ZZ^n(\underline{n})} 
p^{\ell_B(f)}a^{\nneg(f)}\\
&=\sum_{n_0+\cdots+n_k=n}{n \brack n_0, n_1, \ldots, n_k}_{p}\frac{(-ap;p)_{n}}{(-ap;p)_{n_{0}}} q^{\sum_i n_i(k-i)}\\
&=[\hat n]_{a,p}!\times \textrm{coefficient of} \; u^n \; \textrm{in}\;
e[u]_pe[qu]_{p}\ldots e[q^{k-1}u]_{p} {\hat e}[q^ku]_{a,p}.
\end{align*}
Using the formula
\begin{equation}\label{tmax}
(1-t)\sum_{k\geq 0}a_{\leq k}t^k=\sum_{k\geq 0}(a_{\leq k}-a_{\leq k-1})t^k=
\sum_{k\geq 0} a_{k}t^k,
\end{equation}
where $a_{\leq k}=a_0+\cdots +a_k$,
we derive immediately
$$
\sum_{f\in \ZZ^n}t^{\max(f)}q^{\max(f)\cdot n-|f|}
p^{\ell_B(f)}a^{\nneg(f)}=(1-tq^n)\sum_{k\geq 0}t^k
\sum_{f\in \ZZ^n \mid \max(f)\leq k} q^{k n-|f|}
p^{\ell_B(f)}a^{\nneg(f)}.
$$
Therefore
\begin{equation}\label{LHS}
\sum_{n\geq 0}\frac{u^n}{[\hat n]_{a,p}!(1-tq^n)}\sum_{f\in \ZZ^n}t^{\max(f)}q^{\max(f)\cdot n-|f|}
p^{\ell_B(f)}a^{\nneg(f)}=
\sum_{k\geq 0}t^k\prod_{j=0}^{k-1}e[q^ju]_{p}\cdot {\hat e}[q^ku]_{a,p}.
\end{equation}
\bigskip

On the other hand, from the bijection $\psi$, and Equations (\ref{eq:1}) and (\ref{eq:2}) in Proposition~\ref{p:psi}, it follows that 
\begin{align}
\sum_{f\in \ZZ^n \mid \pi(f)=\beta}
t^{\max(f)}q^{|f|}&=\sum_{\lambda}t^{\max(\lambda)+\des_B(\be)}q^{|\lambda|+n\des_B(\be)-\maj(\be)}\nonumber \\
&=\frac{t^{\des_B(\be)}q^{\sum_{i\in \Des_B(\be)}(n-i)}}{(tq;q)_n}. \label{seq}
\end{align}
Replacing $q$ by $q^{-1}$ and $t$ by $tq^n$ in (\ref{seq}) we get 
\begin{align}\label{caseB}
\sum_{f \in \ZZ^n \mid \pi(f)=\be} t^{\max(f)}q^{\max(f)\cdot n-|f|}
=\frac{t^{\des_B(\be)}q^{\maj(\be)}}{(t;q)_n}.
\end{align}
Hence
\begin{align}
\sum_{f\in \ZZ^n} t^{\max(f)}q^{\max(f)\cdot n-|f|}
p^{\ell_B(f)}a^{\nneg(f)} \nonumber
&=\sum_{\be \in B_n}p^{\ell_B(\be)}a^{\nneg(\be)}
\sum_{f \in \ZZ^n \mid \pi(f)=\be} t^{\max(f)}q^{\max(f)\cdot n-|f|}\\
&=\sum_{\be \in B_n}\frac{t^{\des_B(\be)}q^{\maj(\be)}p^{\ell_B(\be)}a^{\nneg(\be)}}{(t;q)_n}.\label{RHS}
\end{align}
By comparing (\ref{LHS}) and (\ref{RHS}) the result follows.
\end{proof}

\begin{remark}[Chow-Gessel's identities]
By letting $p=1$, and substituting $q\leftarrow q^2$, $a\leftarrow q$, and $u \leftarrow (1+q)u$ in \eqref{eq:main1} we obtain Chow-Gessel's formula \eqref{eq:cg-carlitzbis}. 
Letting $p=1$  and substituting $u$ by $(1+a)u$ in \eqref{eq:main1}, then 
extracting the coefficient of $u^n/n!$ 
yields another result of Chow-Gessel~\cite[Equation (26)]{cg07}
\begin{align}\label{eq:cg2}
\frac{\sum_{\be \in B_n}t^{\des_B(\be)}q^{\maj(\be)}a^{\nneg(\be)}}{(t;q)_{n+1}}
=\sum_{k\geq 0}([k+1]_q+a[k]_q)^nt^k.
\end{align}
This result also follows  from \eqref{eq:reiner1} and Proposition~\ref{p:connection}. Indeed, substituting $a\leftarrow aq^{-1}$, $p\leftarrow 1$ and $u\leftarrow (1+aq^{-1})u$ in 
\eqref{eq:reiner1} we get 
\begin{align}\label{eq:reiner3}
 \sum_{n\geq 0}\frac{ \sum_{\be\in B_n}t^{d(\be)}q^{\maj_R(\be)-
 \nneg(\beta)}a^{\nneg(\beta)}}{ (t;q)_{n+1} }\frac{u^n}{n!}
=\sum_{k\geq 0}t^k\exp\left([k+1]_q+a[k]_q\right)u.
\end{align}
In view of Proposition~\ref{p:connection},  Equation \eqref{eq:cg2}  follows then by extracting the coefficient of $u^n/n!$ in \eqref{eq:reiner3}.

Since the right-hand side of  \eqref{eq:cg2} is invariant by the substitution 
$q\to 1/q, a\to a/q$ and $t\to tq^n$ we  obtain immediately the following result
$$
\sum_{\be\in B_n}t^{\des_B(\be)} q^{\maj(\be)}a^{\nneg(\be)}=
\sum_{\be\in B_n}t^{\des_B(\be)} q^{n \des_B(\be)-\maj(\be)-\nneg(\be)}a^{\nneg(\be)}.
$$
\end{remark}

\begin{remark}[Reiner and Brenti's identites]
Comparing the right-hand sides of \eqref{eq:reiner1} and \eqref{eq:main1} we conclude that the distributions of the two quadruples of statistics are different. However,  if we set $q=1$ in \eqref{eq:main1} we obtain  
\begin{align}\label{e:tridistribution}
 \sum_{n\geq 0}\frac{u^n}{(1-t)^{n+1}}\frac{\sum_{\be \in B_n}t^{\des_B(\be)}p^{\ell_B(\be)}a^{\nneg(\be)}}{[\hat n]_{a,p}!}
=\frac{1}{1-te[u]_p}\cdot {\hat e}[u]_{a,p},
\end{align}
and comparing with Reiner's equation \eqref{eq:reiner1} with $q=1$ we see  that 
 $(d_{R}, \inv_{R}, \nneg)$ and $(\des_{B}, \ell_{B},\nneg)$ are equidistributed over $B_{n}$.
By letting $p=1$ in \eqref{e:tridistribution}   we recover a formula of Brenti~\cite[(14)]{fb94}:
\begin{align*}
\sum_{n\geq 0}\sum_{\be\in B_n}
t^{\des_B(\be)} a^{\nneg(\be)}   \frac{u^n}{n!}=\frac{(1-t)e^{u(1-t)}}{1-te^{(1+a)(1-t)u}}.
\end{align*}
\end{remark}

\begin{remark}[Gessel-Roselle identity for $B_n$]
To compute the generating function of major index and length we proceed as follows.
Setting $a=1$ in equation (\ref{eq:main1}) yields
\begin{equation}\label{maj-inv}
\sum_{n\geq 0}\frac{(1-p)^n u^n}{(t;q)_{n+1}(p^2;p^2)_n}\sum_{\be\in B_n}
t^{\des_B(\be)}q^{\maj(\be)}p^{\ell_B(\be)}=
\sum_{k\geq 0}t^ke[u]_pe[qu]_{p}\cdots e[q^{k-1}u]_{p}{\hat e}[q^ku]_{1,p}.
\end{equation}
By multiplying both sides of (\ref{maj-inv}) by $(1-t)$, and then by sending $t\rightarrow 1$ we obtain
\begin{eqnarray*}
\sum_{n\geq 0}\frac{((1-p)u)^n}{(q;q)_{n}(p^2;p^2)_n}\sum_{\be\in B_n}
q^{\maj(\be)}p^{\ell_B(\be)} & = &
\displaystyle{\prod_{i\geq 0}e[q^iu]_{p}}.
\end{eqnarray*}
Replacing  $u$ by $u/(1-p)$ and applying $q$-binomial formula 
$e[u/(1-p)]_p= \prod_{j \geq 0} \frac{1}{1-p^ju}$
we get the following  $B_n$-analogue  of an identity of Gessel-Roselle (see 
\cite[Theorem 8.5]{ge-thesis}  and the historical note  after Theorem 4.3 in  \cite{agr05})
\begin{align}\label{eq:Bge}
 \sum_{n\geq 0}\frac{\sum_{\be\in B_n}
q^{\maj(\be)}p^{\ell_B(\be)}}{(q;q)_{n} (p^2;p^2)_n} u^n=
\prod_{i,j\geq 0} \frac{1}{1-u p^i q^j}.
\end{align}
We refer the reader to \cite[Proposition 5.4]{bi08} for a different generalization of this identity.
\end{remark}
\section{The $D_n$ case.}

The aim of this section is to obtain a generating series for the four-variate distribution of descents, major index, length and number of negative entries over the group $D_n$ by using the encoding of Section~\ref{bijection1}.
We will see that this time we are not able to get a nice identity as in the $B_n$-case; this is not so surprising as we explain in Remark~\ref{r:final}.
\bigskip

Let $\ZZ^{n}_{e}$ be the set of sequences from $[n]$ to $\ZZ$ with an even number of negative entries.
For $f\in \ZZ^{n}_{e}$  Definition~\ref{d:statistics} is still valid and 
we let 
$$\ell_D(f):=\ell_D(\pi(f)),
$$
 where $\pi(f)\in D_{n}$ is the even signed permutation defined 
in Proposition~\ref{d:bijection}.

For  $\underline{n}=(n_0,n_1,\ldots,n_k)$ a composition of $n$, we
denote  $\ZZ^{n}_{e}(\underline{n})=\{f \in \ZZ^{n}_{e} \mid  \#\{i:|f_i|=j\}=n_j\}$.  We have the following lemma.

\begin{lemma}
Let $\underline{n}=(n_0,n_1,\ldots,n_k)$ be a composition of $n$. Then 
\begin{eqnarray*}
\sum_{f \in  \ZZ^{n}_{e}(\underline{n})} p^{\ell_D(f)}a^{\nneg(f)} 
&=& \frac{1}{2}{n \brack n_0,n_1,\ldots, n_k}_p \cdot \left\{\frac{(a;p)_n}{(a;p)_{n_0}}+\frac{(-a;p)_n}{(-a;p)_{n_0}}\right\}.
\end{eqnarray*}
\end{lemma}
\begin{proof}
It is clear that 
$$ \sum_{f \in  \ZZ^{n}_{e}(\underline{n})} p^{\ell_D(f)}a^{\nneg(f)} = \frac{1}{2}
\left[\sum_{f \in  \ZZ^{n}(\underline{n})} p^{\ell_D(f)}a^{\nneg(f)} + 
       \sum_{f \in  \ZZ^{n}(\underline{n})} p^{\ell_D(f)}(-a)^{\nneg(f)}\right].$$
By definition we have that $\ell_B(\gamma)=\ell_D(\g)+\nneg(\g)$, so by \eqref{lemmaReiner} with $a \leftarrow a/p$ we get 
$$\sum_{f \in  \ZZ^{n}(\underline{n})} p^{\ell_D(f)}a^{\nneg(f)} ={n \brack n_0,n_1,\ldots, n_k}_p \cdot \frac{(-a;p)_n}{(-a;p)_{n_0}}.$$
The result follows by replacing $a$ by $-a$ in the above formula.
\end{proof}

In order to obtain our identity we want to compute the following generating series in two different ways 
$$
\sum_{f\in \ZZ^{n}_{e}}t^{ \max(f)} q^{\max(f)\cdot n-|f|}
p^{\ell_D(f)}a^{\nneg(f)}.
$$ 

First of all we have
\begin{align*}
&\sum_{f\in \ZZ^{n}_{e} \mid \max(f)\leq k} q^{k\cdot n-|f|}
p^{\ell_D(f)}a^{\nneg(f)} =\sum_{n_0+\cdots+n_k=n}q^{\sum_i n_i(k-i)}\sum_{f\in \ZZ^n_e} 
p^{\ell_D(f)}a^{\nneg(f)}\\
&=\frac{1}{2} \sum_{n_0+\cdots+n_k=n} \frac{[n]_p!}{[n_0]_p!\ldots [n_k]!_p} q^{\sum_i n_i(k-i)}\left\{
\frac{(a;p)_n}{(a;p)_{n_0}}+\frac{(-a;p)_{n}}{(-a,p)_{n_0}}\right\}\\
&=
\frac{1}{2} [n]_p! \langle{u^n}\rangle \left(e[u]_p e[uq]_p\cdots e[uq^{k-1}]_p\left\{(a;p)_n\hat{e}[uq^k]_{-a/p;p}+(-a;p)_n\hat{e}[uq^k]_{a/p;p}\right\}\right),
\end{align*}
where the notation $\langle{u^n}\rangle f(u)$ stands for the coefficient of $u^n$ in $f(u)$. Using the formula (\ref{tmax}) we derive immediately
$$
\sum_{f\in \ZZ^{n}_{e}}t^{ \max(f)} q^{\max(f)\cdot n-|f|}
p^{\ell_D(f)}a^{\nneg(f)}=(1-tq^n)\sum_{k\geq 0}t^k \sum_{f \in \ZZ^{n}_{e} \mid \max(f)\leq k} q^{k\cdot n-|f|} p^{\ell_D(f)}a^{\nneg(f)}.
$$
Therefore we obtain
\begin{align}
&\frac{\sum_{f\in \ZZ^{n}_{e}} t^{\max(f)}q^{\max(f)\cdot n-|f|}
p^{\ell_D(f)}a^{\nneg(f)}}{(1-tq^n)[n]_p! } \nonumber \\
&=\frac{1}{2} \sum_{k\geq 0} t^k \langle{u^n}\rangle \left\{e[u]_p e[uq]_p\cdots e[uq^{k-1}]_p\left(
(a;p)_n\hat{e}[uq^k]_{-a/p;p}+(-a;p)_n\hat{e}[uq^k]_{a/p;p}\right)\right\}. \label{LHSD}
\end{align}

On the other hand, we decompose $D_n$ as the disjoint union of the following three sets
\begin{align}
D_n^{+}&:=\{\g \in D_n \mid 0\not\in \Des_D(\g) \ {\rm and} \ \g(1)<0 \},\\
D_n^{-}&:=\{\g \in D_n \mid 0 \in \Des_D(\g) \ {\rm and} \ \g(1)>0 \},\\
D_n^{0}&:=D_n \setminus (D_n^{+} \cup D_n^{-}).
\end{align}

We have the following lemma.
\begin{lemma} \label{lem:d} For $\g \in D_{n}$ we have
$$\sum_{f\in \ZZ^{n}_e \mid \pi(f)=\g} t^{\max(f)} q^{\max(f) \cdot n-|f|}
=\frac{D_n^\gamma(t,q)}{(t;q)_n},$$
where
$$
D_n^\gamma(t,q):=\left\{\begin{array}{ll}
t^{\des_D(\g)+1}q^{\maj(\g)}, & \ {\rm if} \ \g \in D_n^{+} \\
t^{\des_D(\g)-1}q^{\maj(\g)}, & \ {\rm if} \ \g \in D_n^{-}\\ 
t^{\des_D(\g)}q^{\maj(\g)},   & \ {\rm if} \ \g \in D_n^{0}.
\end{array}
\right.
$$
\end{lemma}
\begin{proof}
The proof easily follows from Equation~(\ref{caseB}). Let show the first case. 
Suppose that $\pi(f)=\gamma\in D_{n}^{+}$, that is  $0 \not \in \Des_D(\g)$ and $\g(1)<0$. Then
\begin{eqnarray*}
\sum_{f\in \ZZ^{n}_e \mid \pi(f)=\g} t^{\max(f)} q^{\max(f) \cdot n-|f|}= \sum_{f\in \ZZ^{n} \mid \pi(f)=\g} t^{\max(f)} q^{\max(f) \cdot n-|f|}
& = & \frac{t^{\des_B(\gamma)}q^{\maj(\gamma)}}{(t;q)_n} \\
& = & \frac{t^{\des_D(\gamma)+1}q^{\maj(\gamma)}}{(t;q)_n}.
\end{eqnarray*}
In fact, $0 \in \Des_B(\g)$ since $\g(1)<0$, but $0 \not\in \Des_D(\g)$. The other cases are similar.
\end{proof}

For $\gamma\in D_n$  set
$$
w(\gamma):=t^{\des_D(\g)}q^{\maj(\g)}p^{\ell_D(\g)}a^{\nneg(\g)}.
$$
It follows from  Lemma~\ref{lem:d} that 
\begin{align*}
\sum_{f\in \ZZ^n_e} t^{\max(f)}q^{\max(f)\cdot n-|f|}
p^{\ell_D(f)}a^{\nneg(f)} &=\sum_{\g\in D_n}p^{\ell_D(\g)}a^{\nneg(\g)}
\sum_{f\in \ZZ^n_e \mid \pi(f)=\g} t^{\max(f)}q^{\max(f)\cdot n-|f|}  \\
&=\frac{1}{(t;q)_n}\left(\sum_{\g \in D_n^{0}}w(\gamma)+
\frac{1}{t} \sum_{\g \in D_n^{-}}w(\gamma)
+ t \sum_{\g \in D_n^{+}}w(\gamma)\right).
\end{align*}
\smallskip
By comparing the above equation with 
 Equation \eqref{LHSD}  we obtain the following identity.
\begin{theorem}\label{identityTD}
\begin{align}\label{identityD}
&\sum_{\g \in D_n^{0}}\frac{w(\gamma)}{(t;q)_n}+
\frac{1}{t} \sum_{\g \in D_n^{-}}\frac{w(\gamma)}{(t;q)_n}
+ t \sum_{\g \in D_n^{+}}\frac{w(\gamma)}{(t;q)_n} \nonumber \\
&=\frac{1}{2}\sum_{k\geq 0} t^k \langle{u^n}\rangle \left\{e[u]_p e[uq]_p\cdots e[uq^{k-1}]_p\left(
(a;p)_n\hat{e}[uq^k]_{-a/p;p}+(-a;p)_n\hat{e}[uq^k]_{a/p;p}\right)\right\}.
\end{align}
\end{theorem}

If $a=1$  and letting $t\to 1$ in the above identity we obtain the $D_n$-analogue of
Gessel-Roselle's identity (see \eqref{eq:Bge}):
\begin{align}
1+\sum_{n\geq 1}u^n\frac{\sum_{\g\in D_n}q^{\maj(\g)}p^{\ell_D(\g)}}{(-p;p)_{n-1}(q;q)_{n-1}}=
\prod_{i,j\geq 0}\frac{1}{1-(1-p)up^iq^j}.
\end{align}
\smallskip

\section{Final Remarks}

\begin{remark}\label{r:final}
Note that Brenti~\cite[Proposition 4.3]{fb94} computed the generating function of $(\des_{D}, \nneg)$, and 
Reiner~\cite[Theorem 7]{re95} that of $(\des_{D}, \ell_{D})$ over $D_{n}$. Their formulas are much more involved than the corresponding $B_{n}$-cases, even in the one-variable case, see Equation~\eqref{e:brenti}. 
Reiner gave also a method to compute the  distribution of $(\des_{D}, \ell_{D}, \nneg)$ over $D_{n}$.
However, it doen's seem that his method allows to include the statistic major index in the computation.
\end{remark}

\begin{remark}
The restriction of the bijection $\psi$ (defined in Proposition~\ref{p:psi}) to $\ZZ^{n}_{e}$ is not a well defined map from $\ZZ^{n}_{e}$ to $D_{n}\times {\mathcal P}_{n}$. For example, consider $f=(0, -3, -4)\in \ZZ^{3}_{e}$. Then $\pi(f)=[1,-2,-3]\in D_{3}$, and $\la(f)=(-1, 1, 1)$ which is not a partition. Hence the following question naturally arises. 
\end{remark}

\begin{question} Is there a parametrization of the elements of $\ZZ_e^n$ which reduces the LHS of Equation \eqref{identityD} to  a single sum with a simpler RHS ?
\end{question}

The desired equation should coincide for $q=p=a=1$ with the following result of Brenti~\cite[Theorem 4.10]{fb94} for the Eulerian polynomials of type $D$:
\begin{equation}\label{e:brenti}
\frac{\sum_{\g \in D_n} t^{\des_D(\g)}}{(1-t)^{n+1}}=\sum_{k\geq 0}\{(2k+1)^n-n2^{n-1}[\mathcal{B}_n(k+1)-\mathcal{B}_n]\}t^k,
\end{equation}
where $\mathcal{B}_n(x)$ is the $n$th Bernoulli polynomial and $\mathcal{B}_n$ the $n$th Bernoulli number. See also \cite[\S 5]{cg07}. Recently Mendes and Remmel~\cite{mr08} computed some generating series over $B_n$ and $D_n$, closely related to ours. Unfortunately, their computations don't answer the above question.  

\begin{remark}
In a forthcoming paper~\cite{bz} we will study the distribution of several statistics over the wreath product of a
symmetric group with a cyclic group.  Among our results,  we will give an extension of Theorem~\ref{t:main1}.
\end{remark}


\end{document}